\documentclass[a4paper, 11pt, reqno]{amsart}
\usepackage[usenames,dvipsnames]{color}
\usepackage{amssymb, amsmath, latexsym, graphics, graphicx, pstricks}

\newcommand\GreenR{\mathcal{R}}
\newcommand\GreenH{\mathcal{H}}

\newcommand\trop{\mathbb{T}}

\newcommand\ft{\mathbb{FT}}

\newcommand\pft{\mathbb{PFT}}

\newtheorem{theorem}{Theorem}[section]
\newtheorem{lemma}[theorem]{Lemma}
\newtheorem{proposition}[theorem]{Proposition}

\newtheorem{corollary}[theorem]{Corollary}

\begin{document}
\title[Idempotent tropical matrices and finite metric spaces]{Idempotent tropical matrices and \\ finite metric spaces}
\maketitle

\begin{center}

    MARIANNE JOHNSON\footnote{Email \texttt{Marianne.Johnson@maths.manchester.ac.uk}.}
and
 MARK KAMBITES\footnote{Email \texttt{Mark.Kambites@manchester.ac.uk}.}

    \medskip

    School of Mathematics, \ University of Manchester, \\
    Manchester M13 9PL, \ England.

\keywords{}
\thanks{}

\end{center}

\numberwithin{equation}{section}
\begin{abstract}
There is a well known correspondence between the triangle inequality for a distance function on a finite set, and
idempotency of an associated matrix over the tropical semiring. Recent research has shed new light
on the structure  (algebraic, combinatorial and geometric) of tropical idempotents, and in this paper
we explore the consequences of this for the metric geometry of tropical polytopes.
We prove, for example, that every $n$-point metric space is realised by the Hilbert projective metric on the vertices
of a pure $n$-dimensional tropical polytope in tropical $n$-space. More generally, every $n$-point
asymmetric distance function is realised by a residuation operator on the vertices of
such a polytope. In the symmetric case, we show that the maximal group of tropical matrices containing the
idempotent associated to a metric space is isomorphic to $G \times \mathbb{R}$, where $G$ is the isometry group
of the space; it follows that every group of the form $G \times \mathbb{R}$ with $G$ finite arises as a maximal
subgroup of a sufficiently large finitary full tropical matrix semigroup. In the process we also prove some new results
about tropical idempotent matrices, and note some semigroup-theoretic consequences which
may be of independent interest.
\end{abstract}

\section{Introduction}
Recently there has been increasing interest in using tropical methods in
finite metric geometry.
Given a finite ordered set $X$ and a function $d : X \times X \to \mathbb{R}$ satisfying $d(x,x) = 0$
for all $x$, we may consider the $|X| \times |X|$ matrix $D$ whose entries are given by the
function $-d$ as a matrix over the max-plus semiring. It is well known (see for example
\cite{Develin04}) that $d$ satisfies the triangle inequality if and only if the matrix is
\textit{idempotent}, that is, $D \otimes D = D$ as max-plus matrices. Hence it is easy to see
that $d$ is a metric if and only if $D$ is idempotent, non-positive and
symmetric with zeros exactly on the diagonal.

There have also been significant recent advances in understanding the algebraic structure of tropical
matrices and polytopes. In particular, work of Izhakian and the present authors \cite{K_puredim} has
yielded new insight into the properties of projective tropical polytopes, and hence of tropical
idempotent matrices. Our main aim in the present paper is to put this insight to work by studying
tropical representations of finite metric spaces. For example, we shall see that every $n$-point
finite metric space can be realised as the Hilbert projective metric on the vertex set of a pure $n$-dimensional
tropical (and Euclidean) convex polytope in tropical projective $(n-1)$-space (a \textit{polytrope}, in
the language of \cite{Joswig10}). This gives a stark contrast between tropical and Euclidean convex geometry, since
there are metric spaces on four points which cannot be embedded into Euclidean space of any dimension (see
Section~\ref{sec_examples} below for an example).

Our results also have consequences for the theory of tropical matrix semigroups and groups. Recent work of Izhakian
and the authors \cite{K_tropgroups} has shown that every maximal subgroup of the $n \times n$ finitary tropical matrix
semigroup has the form $G \times \mathbb{R}$ for some finite group $G$. We show that the maximal subgroup around the
idempotent associated to any finite metric space is naturally isomorphic to $G \times \mathbb{R}$ where $G$ is the isometry
group of the space. Since every finite group is the isometry group of a finite metric space \cite{Asimov76}, this means that
the maximal subgroups of all full square tropical matrix semigroups are exactly the groups of the form $G \times \mathbb{R}$
with $G$ finite.

From an abstract algebraic viewpoint, idempotency of a tropical matrix is an extremely natural condition,
but symmetry seems perhaps a more artificial imposition. Removing the requirement for symmetry leads
(modulo some technicalities, to be described below) to matrices representing asymmetric distance
functions, which we shall term \textit{semimetrics}\footnote{Terminology for the various possible
generalisations of a metric is not standardised; we caution that the functions we consider are sometimes termed by others ``premetrics'' or ``quasi-metrics'', while the term ``semimetric'' is used by some authors for a distance
function with the triangle inequality relaxed.}. Such functions are clearly ubiquitous in nature, and
in particular occur in many areas of applied mathematics. In recent years they
have also begun to emerge more often in pure mathematics (see for example \cite{K_semimetric}). To date,
however, they have not achieved prominence as objects of pure mathematical study, and there is no
coherent subject of ``asymmetric geometry''. We believe this reflects not a lack of importance, but
rather a lack of effective methods. It transpires that the tropical representation of finite metrics
can be extended to semimetrics, by replacing the
Hilbert projective metric on projective tropical $(n-1)$-space with a suitable
\textit{residuation operator}
on \textit{affine} tropical $n$-space. We believe this may provide a
useful tool for studying semimetrics.

In addition to this introduction, this article comprises seven sections.
In Section~\ref{sec_prelim} we recap some foundational definitions and
summarise some required results from \cite{K_puredim}. Section~\ref{sec_idpt} establishes some basic
facts about the structure of idempotent matrices.
In Section~\ref{sec_semimetric} we give characterisations of finite metrics and semimetrics in
terms of geometric properties of the associated idempotent matrices.
Section~\ref{sec_residuation} shows that every finite semimetric on $n$
points can be embedded into affine tropical $n$-space.
Section~\ref{sec_duality} discusses the relationship between finite metric spaces and tropical polytopes
that are convex in the ordinary sense, and also derives some semigroup-theoretic consequences.
Section~\ref{sec_groups} proves our results about maximal subgroups of full tropical matrix semigroups.
 Finally,
Section~\ref{sec_examples} studies some low-dimensional examples which illustrate our results; these are
collected at the end for ease of reference to diagrams, but the reader may also wish to consult them while reading the earlier
sections.

\section{Preliminaries}\label{sec_prelim}

We write $\ft$ for the set $\mathbb{R}$ equipped with the operations of maximum (denoted by $\oplus$) and addition (denoted by $\otimes$, by $+$ or
simply by juxtaposition). Thus, we write $a \oplus b = \max(a,b)$ and $a \otimes b = ab = a + b$. It is readily verified that $\ft$ is an abelian
group (with neutral element $0$) under $\otimes$ and a commutative semigroup of idempotents (without a neutral element) under $\oplus$, and
that $\otimes$ distributes over $\oplus$. These properties mean $\ft$ has the structure of an \textit{idempotent semifield} (without zero).

Let $M_n(\ft)$ denote the set of all $n \times n$ matrices with entries in $\ft$. The operations $\oplus$ and $\otimes$ can be extended in the obvious way to give corresponding operations on $M_n(\ft)$. (In particular, it is easy to see that $M_n(\ft)$ is a semigroup with respect to tropical matrix multiplication.)

We shall be interested in the space $\ft^n$ consisting of $n$-tuples $x$ with entries in $\ft$; we
write $x_i$ for the $i$th component of $x$. We call $\ft^n$ \textit{(affine) tropical $n$-space}.
The space $\ft^n$ admits
an addition and a scaling action of $\ft$ given by $(x\oplus y)_i = x_i \oplus y_i$ and
$(\lambda x)_i = \lambda (x_i)$ respectively. These operations give $\ft^n$
the structure of an \textit{$\ft$-module}\footnote{Some authors use the term \textit{semimodule}, to
emphasise the non-invertibility of addition, but since no other kind of module exists over $\ft$
we have preferred the more concise term.}. It also
has the structure of a lattice, under the partial order given by $x \leq y$ if $x_i \leq y_i$ for all $i$.

From affine tropical $n$-space we obtain \textit{projective tropical $(n-1)$-space}, denoted $\pft^{n-1}$,
by identifying two vectors if one is a tropical multiple of the other by an element of $\ft$. We
identify $\pft^{n-1}$ with $\mathbb{R}^{n-1}$ via the map
$$(x_1, \ldots, x_n) \mapsto (x_1-x_n, x_2 -x_n, \ldots,  x_{n-1} - x_n).$$

Submodules of $\ft^n$ (that is, subsets closed under tropical addition and scaling) are termed
\textit{(tropical) convex sets}. Finitely generated convex sets are called \textit{(tropical) polytopes}.
Since convex sets are closed under scaling, each convex set $X \subseteq \ft^n$ induces a subset of
$\pft^{n-1}$, termed the \textit{projectivisation} of $X$.

For $A \in M_n(\ft)$ we let $R(A)$ denote the tropical polytope in $\ft^n$ generated by the rows of $A$ and let $C(A)$ denote the tropical polytope in $\ft^n$ generated by the columns of $A$. (In the interest of brevity we shall ignore the distinction between row and column vectors,
regarding $R(A)$ and $C(A)$ as submodules of the same space $\ft^n$, whose elements will be written in the form $v = (v_1, \ldots, v_n)$.) We call these tropical polytopes the \emph{row space} and \emph{column space} of $A$ respectively.

A point $x$ in a convex set $X$ is called \textit{extremal in $X$} if the set
$$X \smallsetminus \lbrace \lambda \otimes x: \lambda \in \ft \rbrace$$
is a submodule of $X$. Clearly some scaling of every such extremal point must lie in every generating set for
$X$. In fact, every tropical polytope is generated by its extremal points considered up to scaling \cite{Butkovic07,Wagneur91}.

There are several important notions of dimension for convex sets. The \emph{tropical dimension} is the
topological dimension of the set, viewed as a subset of $\mathbb{R}^n$ with the usual topology. Note
that, in contrast to the classical (Euclidean) case, tropical convex sets may have regions of different
topological dimension. We say a set $X$ has \emph{pure} dimension $k$ if every open (within X with the
induced topology) subset of X has topological dimension $k$. The \emph{generator dimension} of a convex
set $X$ is the minimal cardinality of a generating subset, under the linear operations of scaling and
addition. (If $X$ is a polytope, this is equal to the number of extremal points of $X$ considered up to scaling.) The \emph{dual dimension}~\cite{K_puredim} is the minimal cardinality of a generating set under scaling and
the induced operation of greatest lower bound within the convex set. (Notice that, in general, the greatest lower bound of two elements within a convex set $X$ need not be the same as their component-wise minimum, which may not be contained in $X$.)

In \cite{K_puredim}, Izhakian and the present authors gave a characterisation of
\emph{projectivity} for tropical polytopes in terms of the geometric and order-theoretic structure on these sets. We briefly recall that a module $P$ is called \emph{projective} if every morphism from $P$ to another module $M$ factors through every surjective module morphism onto $M$. One of the
main results of \cite{K_puredim} can now be summarised as follows.

\begin{theorem}\cite[Theorems 1.1 and 4.5]{K_puredim}.
\label{IJKmain}
Let $X\subseteq \ft^n$ be a tropical polytope. Then the following are equivalent:
\begin{itemize}
\item[(i)] $X$ is projective as an $\ft$-module;
\item[(ii)] $X$ is the column space of an idempotent matrix in $M_n(\ft)$;
\item[(iii)] $X$ has pure dimension equal to its generator dimension and dual dimension.
\end{itemize}
\end{theorem}

Since all three notions of dimension coincide for projective polytopes, we define the \emph{dimension} of a projective tropical polytope to be this common value. We shall refer to projective polytopes of dimension $k$ as \emph{projective $k$-polytopes}. Projective $n$-polytopes in $\ft^n$ turn out to have a particularly nice structure:

\begin{theorem}\cite[Proposition 5.5]{K_puredim}.
\label{minplus}
Let $X\subseteq \ft^n$ be a projective $n$-polytope. Then $X$ is min-plus (as well as max-plus) convex.
\label{npolytrope}
\end{theorem}

It is easily verified that any tropical polytope that is min-plus (as well as max-plus) convex must be
convex in the usual (Euclidean) sense. Theorem~\ref{minplus} thus says that projective $n$-polytopes
in $\ft^n$ are \textit{polytropes} in the sense of Joswig and Kulas \cite{Joswig10}.

Numerous definitions of rank have been introduced and studied for tropical
matrices, mostly corresponding to different notions of ``dimension'' of the
row or column space. In light of the previous theorem, we shall focus on the
following three definitions of rank. The \emph{tropical rank} of a matrix is the tropical dimension of its row space (or equivalently, by \cite[Theorem 23]{Develin04} for example, its column space). It can be shown
\cite{Develin05} that the tropical rank is also the largest positive integer
$k$ such that there is a $k \times k$ minor whose permanent is attained by a
unique permutation $\sigma \in S_k$. The \emph{row rank} is the generator
dimension of the row space, which by \cite[Proposition 3.1]{K_puredim} is
also the dual dimension of the column space. Dually, the \emph{column rank}
is the generator dimension of the column space and also the dual
dimension of the row space. We remark that other notions of rank for tropical
matrices are also studied; see for example \cite{Akian06,Develin05} for more details.

\section{Structure of Tropical Idempotents}\label{sec_idpt}

In this section we study the structure of idempotent matrices over $\ft$. We begin with the
observation that, while the notions of rank described in Section~\ref{sec_prelim}
(tropical rank, row rank and column rank) can all differ for tropical matrices in general, it follows easily from
Theorem~\ref{IJKmain} and our remarks above that they all coincide for idempotent matrices. (In fact, it is shown in \cite{K_puredim} that most of the commonly studied notions of rank coincide for idempotent matrices.)
Thus we may refer without ambiguity to the \emph{rank} of an idempotent matrix. Moreover, given an idempotent matrix $E$ of rank $k$, it follows from Theorem~\ref{IJKmain} above that the row space and column space of $E$ are of \emph{pure} topological dimension $k$. In the following sections we shall be particularly interested in idempotent matrices in $M_n(\ft)$ of full tropical rank $n$, often termed \emph{strongly regular} idempotents. For the moment we consider general idempotents.

\begin{lemma}\cite[Lemma 5.2]{K_puredim}.
\label{extremalsandzeros}Let $E$ be an idempotent element of $M_n(\ft)$. Then
every extremal point of the column [row] space of $E$ occurs up to scaling as
a column [row] of $E$ with diagonal entry $0$.
\end{lemma}

Thus the columns [rows] of $E$ with zero diagonal entry generate the column [row] space of $E$. This gives an upper bound for the rank of an idempotent matrix.

\begin{corollary}
\label{rankandzeros}
Let $E$ be an idempotent element of $M_n(\ft)$. Then the rank of $E$ is less than or equal to the number of zeros on its diagonal. In particular, any strongly regular idempotent has all diagonal entries equal to $0$.
\end{corollary}

We note that we shall see shortly that there are idempotents with all diagonal entries equal to zero which are not strongly regular.

We recall that any matrix $A \in M_n(\ft)$ has a unique eigenvalue, which can be calculated as the maximum average weighted path from a node to itself in the weighted directed graph corresponding to $A$ (see for example \cite{Butkovic10} for details). If this eigenvalue is non-positive then the following series, known as the Kleene star of $A$, converges to a finite limit in $M_n(\ft)$, denoted $A^*$:
$$I_n \oplus A \oplus A^2 \oplus \cdots \oplus A^n \oplus \cdots,$$
where $A^k$ denotes the $k$th power of $A$ in $M_n(\ft)$ and $I_n$ denotes the $n \times n$ matrix whose diagonal entries are $0$ and whose off diagonal entries are equal to $-\infty$ (here we define $-\infty \oplus a = a \oplus -\infty = a$ for all $a \in \ft$). For example, it is clear that any idempotent matrix $E$ has eigenvalue $0$ and also that
\begin{eqnarray*}
E^* &=& I_n \oplus E \oplus E^2 \oplus \cdots \oplus E^n \oplus \cdots\\
    &=& I_n \oplus E \oplus E \oplus \cdots \oplus E \oplus \cdots\\
    &=& I_n \oplus E.
\end{eqnarray*}

With the (usual) definition that $-\infty \otimes a = a \otimes -\infty= -\infty$ for all $a \in \ft$, it is easy to see that, when defined, $A^*$ is an idempotent all of whose diagonal entries are equal to zero and hence $A^{**} = A^*$. In fact, it is easy to show that the idempotents whose diagonal entries are all equal to zero are precisely those matrices that are equal to their own Kleene star. (Suppose $E$ is an idempotent matrix with all diagonal entries equal to $0$. Since $E$ is idempotent we have $E^* = I_n \oplus E$, as above. Moreover, since all diagonal entries of $E$ are equal to $0$, $E \oplus I_n=E$, giving $E^* = E$.)

Given an idempotent with all diagonal entries equal to $0$, the following result tells us whether it is strongly regular.

\begin{lemma}
\label{fullrank}
Let $E$ be an idempotent element of $M_n(\ft)$ with all diagonal entries equal to $0$. Then $E$ has rank strictly less than $n$ if and only if $E_{i,j} = -E_{j,i}$ for some $i \neq j$.
\end{lemma}

\begin{proof}
Recall that the permanent of $E$ is
$${\rm perm}(E) = \bigoplus_{\sigma \in S_n} E_{1, \sigma(1)} \otimes  \cdots \otimes E_{n, \sigma(n)} $$
and hence ${\rm perm}(E) \geq E_{1, 1} \otimes  \cdots \otimes E_{n, n} = 0$. We first claim ${\rm perm}(E)=0$. Suppose for contradiction that ${\rm perm}(E) > 0$. Then there is a permutation $\sigma \in S_n$ such that $E_{1, \sigma(1)} \otimes  \cdots \otimes E_{n, \sigma(n)}>0$. Write $\sigma$ as a product of disjoint cycles, say $\sigma = \sigma_1\cdots \sigma_l$. Then at least one such cycle, $\sigma_i = (j_1 \ldots j_k)$ say, satisfies $E_{j_1, j_2} \otimes E_{j_2, j_3}\otimes  \cdots \otimes E_{j_{k-1}, j_k}\otimes E_{j_k, j_1}>0$. But then
$$E_{j_1,j_1} = (E^k)_{j_1,j_1} \geq  E_{j_1, j_2} \otimes  E_{j_2, j_3}\otimes  \cdots \otimes E_{j_{k-1}, j_k}\otimes E_{j_k, j_1}>0, $$
contradicting $E_{j_1,j_1} = 0$. Thus ${\rm perm}(E)=0$ and, moreover, for every cycle $(j_1 \ldots j_k) \in S_n$ we must have $E_{j_1, j_2} \otimes  \cdots \otimes E_{j_{k-1}, j_k}\otimes E_{j_k, j_1}\leq 0$.

Now suppose $E$ has rank strictly less than $n$ and hence the permanent of $E$ is not uniquely attained. Thus, there is a non-trivial permutation $\sigma \in S_n$ such that $E_{1, \sigma(1)} \otimes  \cdots \otimes E_{n, \sigma(n)}=0$.  Write $\sigma$ as a product of non-trivial disjoint cycles, say $\sigma = \sigma_1\cdots \sigma_l$. By our remarks above, any such cycle, $\sigma_i = (j_1 \ldots j_k)$ say, satisfies $E_{j_1, j_2} \otimes E_{j_2, j_3}\otimes  \cdots \otimes E_{j_{k-1}, j_k}\otimes E_{j_k, j_1}=0$. Thus,
$$E_{j_1,j_k} = (E^{k-1})_{j_1,j_k} \geq E_{j_1, j_2} \otimes  \cdots \otimes E_{j_{k-1}, j_k} = -E_{j_k, j_1},$$
giving
$$0 = E_{j_1,j_1} = (E^2)_{j_1,j_1} \geq E_{j_1, j_k} \otimes E_{j_k, j_1} \geq  0$$
and hence $E_{j_1, j_k} = -E_{j_k, j_1}$, as required.

Finally, if $E_{i,j} = -E_{j,i}$ for some $i \neq j$, then it is easy to see
that the identity permutation and the transposition $(i,j)$ both attain the maximum in the permanent of $E$, so that $E$ has (tropical) rank strictly less than $n$.
\end{proof}

The following theorem describes the number of idempotents having a given
polytope as their column space; it slightly improves upon results in
\cite{K_puredim}.

\begin{theorem}
\label{numidmpt}
Let $X \subseteq \ft^n$ be a tropical polytope.
\begin{itemize}
\item[(i)] If $X$ is a projective $n$-polytope, then there is a unique
idempotent $E \in M_n(\ft)$ such that $X=C(E)$.
\item[(ii)] If $X$ is a projective $k$-polytope, where $k<n$, then there
are continuum many idempotents $E\in M_n(\ft)$ such that $X=C(E)$.
\item[(iii)] Otherwise, $X$ is not projective and there is no idempotent
$E\in M_n(\ft)$ such that $X=C(E)$.
\end{itemize}
\end{theorem}

\begin{proof}
By Theorem~\ref{IJKmain}, $X$ is projective if and only if $X$ is the
column space of an idempotent in $M_n(\ft)$. Furthermore, Theorem~\ref{IJKmain} says that the projective polytopes are precisely those polytopes having pure tropical dimension equal to their generator dimension and dual dimension. It is clear that this common dimension is bounded above by $n$. Thus a tropical polytope $X \subseteq \ft^n$ is either projective of dimension $k\leq n$ or not projective. In the case where $X$ is not projective, it follows immediately from the preceding remarks that there is no idempotent $E\in M_n(\ft)$ such that $X=C(E)$, so that (iii) holds. It remains to prove (i) and (ii).

If $X$ is a projective $n$-polytope, then $X$ has generator dimension $n$ and, by
\cite[Theorem 5.7]{K_puredim}, there is a unique idempotent with column
space $X$.

It remains to show that if $X$ is a projective $k$-polytope, where $k<n$, then there are continuum many idempotents with column space $X$.
Since $X$ is projective, there is an idempotent $E \in M_n(\ft)$ with
$C(E)=X$ and rank $k<n$.  Let $c_1, \ldots, c_n$ denote the columns of $E$
and define
$$I = \lbrace i \in [n] \mid E_{i,i} = 0 \rbrace.$$
By Lemma~\ref{extremalsandzeros}, the set $\{c_i: i \in I\}$ forms a generating set for $C(E)$.

Since $E$ has rank strictly less than $n$, it follows that some column $c_j$
can be written as a linear combination of columns from the set
$\{c_i: i \in I \smallsetminus \{j\}\}$. (Note that $j$ need not be an element of $I$.) Choose $\lambda<0$ and let
$E(\lambda)$ be the matrix obtained from $E$ by scaling the $j$th column
by $\lambda$. Denote the columns of $E({\lambda})$ by $d_1, \dots, d_n$, so
$d_j = \lambda \otimes c_j$ and $d_i = c_i$ for all $i \neq j$.

Now $C(E(\lambda))=C(E)=X$ and $E(\lambda) \neq E(\mu)$ for all $\mu \leq 0$ with
$\mu \neq \lambda$. We will show that $E(\lambda)$ is an idempotent, hence giving continuum many idempotents with column space $X$.

Since the entries of $E(\lambda)$ do not exceed the corresponding entries of $E$, it is easy to see that for every column $c_i$ of $E$ we have
$$E(\lambda) \otimes c_i \leq E \otimes c_i.$$
Since $E$ is idempotent we have $E \otimes c_i=c_i$, giving $E(\lambda) \otimes c_i \leq c_i$ for all columns $c_i$.

Moreover, for all $i \in I\smallsetminus\{j\}$ we have $E(\lambda)_{i,i} = E_{i,i}=0$ and hence
$$E(\lambda) \otimes c_i = \bigoplus_{l=1}^n E_{l,i} \otimes d_l \geq E_{i,i} \otimes d_i = 0 \otimes c_i = c_i.$$
Thus for all $i \in I\smallsetminus\{j\}$ we have shown $E(\lambda)\otimes c_i = c_i$. Since these columns form a generating set for $C(E)=C(E(\lambda))$, it follows that $E(\lambda) \otimes c = c$ for all $c \in C(E(\lambda))$and hence $E(\lambda)$ is an idempotent with column space $X$.
\end{proof}

We shall need the following fact, which follows from results in \cite{Butkovic10}, and is proved in detail in \cite{K_tropgroups}.

\begin{lemma}\cite[Lemma~7.1]{K_tropgroups}
\label{exterior}
Let $E$ be a strongly regular idempotent in $M_n(\ft)$, and consider
the column space $C(E)$ as a subset of $\mathbb{R}^n$ equipped with the usual
topology. Then left multiplication by $E$ maps all points exterior to $C(E)$
onto the boundary of $C(E)$.
\end{lemma}

\section{Semimetrics and Idempotents}\label{sec_semimetric}
\label{semi}
Let $X$ be a non-empty set and define a function $d: X \times X \rightarrow \mathbb{R}$. We say $d$ is a \emph{semimetric} on $X$ (or equivalently, $X$ is a semimetric space with respect to $d$) if $d$ satisfies the following conditions:
\begin{itemize}
\item[(a)] $d(x,x) = 0$ for all $x \in X$;
\item[(b)] $d(x,y) \leq d(x,z) + d(z,y)$ for all $x,y,z \in X$;
\item[(c)] $d(x,y) \geq 0$ for all $x, y \in X$;
\item[(d)] $d(x,y) \neq 0$ for $x\neq y$.
\end{itemize}
Hence $d$ is a \emph{metric} on $X$ if $d$ is a semimetric on $X$ satisfying the following symmetry condition:
\begin{itemize}
\item[(e)] $d(x,y) = d(y,x)$ for all $x,y\in X$.
\end{itemize}

Throughout this section we consider only semimetrics on $n$ points, where $n \geq 1$. Thus, without loss of generality, we shall assume from now on that $d$ is a real valued function on pairs of elements from the $n$-element set $[n]= \{1, \ldots, n\}$. Given any such function $d: [n] \times [n] \rightarrow \mathbb{R}$, we let $D$ denote the $n \times n$ matrix whose $(i,j)$ entry is $-d(i,j)$. We shall give algebraic and geometric characterisations of the matrices arising in this way from metrics and semimetrics on $n$ points.

We shall need the following lemma, the idea of which is probably well-known
to experts in the field; parts of it appear in \cite{Butkovic10} for example.
Since the precise statement we need does not appear to be in the
literature, we include a brief proof.

\begin{lemma}Let $n$ be a positive integer. Given a function
$$d:[n] \times [n] \rightarrow \mathbb{R},$$
let $D$ denote the $n \times n$ matrix given by $D = (-d(i,j))$. Then the following are equivalent:
\begin{itemize}
\label{triangle}
\item[(i)] $d$ satisfies $d(i,i)=0$ and
$d(i,j) \leq d(i,k) + d(k,j)$ for all $i,j,k \in [n]$;
\item[(ii)] $D$ is an idempotent matrix in $M_n(\ft)$ with all diagonal entries equal to $0$;
\item[(iii)] The Kleene star $D^*$ is defined and equal to $D$.
\end{itemize}
\end{lemma}
\begin{proof} The equivalence of (ii) and (iii) is given by the comments preceding Lemma~\ref{fullrank}. Thus it remains to show the equivalence of (i) and (ii). Suppose $d$ satisfies $d(i,i)=0$ and $d(i,j) \leq d(i,k) + d(k,j)$ for all $i,j,k \in [n]$. Then clearly all diagonal entries of $D$ are equal to $0$ giving
$$(D^2)_{i,j} = \bigoplus_{k=1}^n D_{i,k} \otimes D_{k,j} \geq D_{i,i} \otimes D_{i,j}
= D_{i,i} + D_{i,j} = D_{i,j},$$
whilst the triangle inequality satisfied by $d$ gives
$$(D^2)_{i,j} = \bigoplus_{k=1}^n D_{i,k} \otimes D_{k,j} =\bigoplus_{k=1}^n (D_{i,k} + D_{k,j})
\leq \bigoplus_{k=1}^n D_{i,j} =D_{i,j}.$$
Thus $D$ is idempotent.

Next suppose $D$ is idempotent with all diagonal entries equal to $0$. Then it is immediate that $d(i,i)=0$. Moreover, the idempotency of $D$ gives
\begin{eqnarray*}
d(i,j) = - D_{i,j} &=& - (D^2)_{i,j} \\
&=& - \bigoplus_{k=1}^n D_{i,k} \otimes D_{k,j}\\
&\leq& - (D_{i,k} \otimes D_{k,j}) =  (- D_{i,k}) + (- D_{k,j}) = d(i,k) + d(k,j)
\end{eqnarray*}
for all $i,j,k \in [n]$.
\end{proof}
In what follows, it will be convenient to write $\underline{0}$ to denote the element $(0,\ldots, 0)$ of $\ft^n$. Using Lemma~\ref{triangle} we see that if $d$ is a semimetric on $n$ points, then conditions (a) and (b) guarantee that the resulting $n \times n$ matrix $D$ will be an idempotent with all diagonal entries equal to $0$. The following theorem exactly describes which idempotent matrices arise in this manner.
\begin{theorem}
\label{semimetricmatrix}
Let $n$ be a positive integer. Given a function
$$d:[n] \times [n] \rightarrow \mathbb{R},$$
let $D$ denote the $n \times n$ matrix given by $D = (-d(i,j))$. Then the following statements are equivalent:
\begin{itemize}
\item[(i)] $d$ is a semimetric;
\item[(ii)] $D$ is a strongly regular idempotent with negative entries off the diagonal;
\item[(iii)] $D= D^*$ with negative entries off the diagonal;
\item[(iv)] $D$ is a strongly regular idempotent and $\underline{0}$ is an interior point of the column space of $D$;
\item[(v)] $D$ is a strongly regular idempotent whose columns sum to $\underline{0}$, an interior point of the column space of $D$;
\item[(vii)]$D$ is a strongly regular idempotent and $\underline{0}$ is an interior point of the row space of $D$.
\item[(vi)] $D$ is a strongly regular idempotent whose rows sum to $\underline{0}$, an interior point of the row space of $D$;
\end{itemize}
\end{theorem}
\begin{proof}
We first prove the equivalence of (i), (ii) and (iii) using Lemmas~\ref{triangle} and~\ref{fullrank}.

Suppose $d$ is a semimetric. Then $d$ satisfies conditions (a)-(d) from the definition. As we have seen, conditions (a) and (b) together with Lemma~\ref{triangle} yield that $D$ is an idempotent with all diagonal entries equal to $0$. Moreover, by conditions (c) and (d), the off-diagonal entries of $D$ are all negative. Applying Lemma~\ref{fullrank} now gives $D$ has rank $n$. Thus we have shown that $D$ is a strongly regular idempotent with negative entries off the diagonal.

Suppose $D$ is a strongly regular idempotent with negative entries off the diagonal. Since $D$ is strongly regular all diagonal entries of $D$ must be equal to $0$, by Corollary~\ref{rankandzeros}. Thus, by Lemma~\ref{triangle}, $D= D^*$.

Suppose $D=D^*$ and $D$ has negative entries off the diagonal. It is immediate that $d(i,j) \neq 0$ for $i \neq j$. Lemma~\ref{triangle} also gives $d(i,i)=0$ and $d(i,j) \leq d(i,k) + d(k,j)$ for all $i,j,k \in [n]$. It then follows that $d(i,j) \geq 0$ for all $i, j \in [n]$. Hence $d$ is a semimetric.

To complete the proof, we show the equivalence of (ii), (iv) and (v), the equivalence of (ii), (vi) and (vii) being dual.

Suppose $D$ is a strongly regular idempotent with all off-diagonal entries
negative. By Corollary \ref{rankandzeros}, each diagonal entry of $D$ is
equal to $0$. Thus it is easy to see that $\underline{0}$ is equal to the
tropical sum of the columns of $D$. We must show that this element lies in
the \emph{interior} of $C(D)$, which by \cite[Theorem~6.2.14]{Butkovic10} is
the same as showing that $\underline{0}$ can be written uniquely as a linear
combination of the columns $c_1, \ldots, c_n$ of $D$. Suppose for
contradiction that
$$\underline{0}=  \lambda_1 \otimes c_1 \oplus \cdots \oplus \lambda_n \otimes c_n,$$
for some $\lambda_i\in \ft$ not all equal to zero. Since $c_{i,i}=0$ it is immediate that $\lambda_i \leq 0$ for all $i$. Thus, by supposition, we must have $\lambda_{j} <0$ for some $j \in [n]$ and it is clear that $\lambda_{j} \otimes c_{j}$ does not attain the $j$th co-ordinate. Now choose $k \neq j$ such that $\lambda_{k} \otimes c_{k}$ attains the $j$th co-ordinate. In other words, we have $\lambda_{k} + D_{j, k} = \lambda_{k} \otimes c_{k, j} = 0$, giving $D_{j, k} = - \lambda_k \geq 0$, contradicting that $D$ has negative entries off the diagonal. Thus we conclude that $\underline{0}$ can be written uniquely as a linear combination of the columns $c_1, \ldots, c_n$ of $D$, namely $\underline{0} =c_1 \oplus \cdots \oplus c_n$. Hence (ii) implies (v). That (v) implies (iv) is trivial.

It remains to show that (iv) implies (ii).
Suppose $D$ is a strongly regular idempotent and $\underline{0}$ is an interior point of the column space of $D$. We must show that the off-diagonal entries of $D$ are negative. Since $\underline{0}$ lies in the interior of the column space of $D$, we have that $\underline{0}$ can be written
uniquely as a linear combination of the columns $c_1, \ldots, c_n$ of $D$. Let
\begin{equation}
\label{unique}
\underline{0}=  \lambda_1 \otimes c_1 \oplus \cdots \oplus \lambda_n \otimes c_n
\end{equation}
be this unique expression. Since $c_{i,i}=0$ it is immediate that $\lambda_i \leq 0$ for all $i$. We first claim $\lambda_i = 0$ for all $i$ (from which it follows easily that $D_{i,j} = c_{j,i} \leq 0$ for all $i,j$).

Suppose for contradiction that $\lambda_{i_1} <0$ for some
$i_1 \in [n]$. Then $\lambda_{i_1}\otimes c_{i_1}$ does not
attain the $i_1$ co-ordinate in \eqref{unique}. Choose $i_2 \neq i_1$ such that $\lambda_{i_2} \otimes c_{i_2}$ attains the $i_1$ co-ordinate in \eqref{unique}. Then $\lambda_{i_2} + D_{i_1, i_2} = \lambda_{i_2} \otimes c_{i_2, i_1} = 0$, giving $D_{i_1, i_2} = - \lambda_{i_2} \geq 0$. Now, if $\lambda_{i_2}=0$ we will have $\lambda_{i_2} \otimes c_{i_2}$ attains the maximum in two co-ordinates, contradicting the uniqueness of expression in \eqref{unique}. Thus, we must have $\lambda_{i_2} <0$ and hence $D_{i_1, i_2} = - \lambda_{i_2} > 0$. Since $\lambda_{i_2} <0$ we may repeat the above argument and choose $i_3 \neq i_2$ such that $\lambda_{i_3} \otimes c_{i_3}$ attains the $i_2$ co-ordinate, i.e. so that $D_{i_2, i_3} = - \lambda_{i_3} \geq 0$. By the same reasoning as before we find $\lambda_{i_3} <0$ and hence $D_{i_2, i_3} = - \lambda_{i_2} > 0$. We note that $i_3 \neq i_1$ since
$$D_{i_1,i_3} = (D^2)_{i_1,i_3} \geq D_{i_1,i_2} + D_{i_2,i_3} >0,$$
whilst $D_{i_3,i_3}=0$. So we have found distinct indices $i_1,i_2,i_3$ with $\lambda_{i_1}, \lambda_{i_2}, \lambda_{i_3} <0$ and $D_{i_1, i_2}, D_{i_2, i_3} >0$. Continuing in this manner we obtain a sequence of distinct indices $i_1, \ldots, i_n$ such that $\lambda_{i_1}, \ldots, \lambda_{i_n} <0$ and $\lambda_{i_{k+1}} \otimes c_{i_{k+1}}$ attains the $i_{k}$ co-ordinate for $k = 1\ldots n-1$.  In particular this gives
$$D_{i_1, i_2}, \ldots,  D_{i_{n-1}, i_n} >0.$$
By uniqueness of the expression \eqref{unique} it follows that each of the terms $\lambda_{i} \otimes c_{i}$ cannot attain more than one co-ordinate. Hence $\lambda_{i_{1}} \otimes c_{i_{1}}$ must attain the $i_n$ co-ordinate, giving $D_{i_n, i_1} = - \lambda_{i_1} > 0$. But then
$$0 = D_{i_1,i_1} = (D^n)_{i_1,i_1} \geq D_{i_1,i_2} + D_{i_2,i_3} + \cdots + D_{i_{n-1},i_n} + D_{i_n,i_1}  >0.$$
Thus we conclude that all $\lambda_i$ in \eqref{unique} are equal to zero. Thus
\begin{equation}
\label{unique0}
\underline{0}= c_1 \oplus \cdots \oplus  c_n
\end{equation}
is the unique expression of $\underline{0}$ as a linear combination of the columns of $E$. It follows immediately that $D_{i,j} = c_{j,i} \leq 0$ for all $i,j$. It
only remains to show that $D_{i,j} <0$ whenever $i \neq j$. Suppose for contradiction that $D_{i,j}=0$ for some $i \neq j$. Then column $j$ contains a zero in position $i$ and position $j$, giving
$$ \underline{0} = c_1 \oplus \cdots \oplus c_{i-1} \oplus (\lambda \otimes c_i) \oplus c_{i+1} \oplus \cdots \oplus c_n,$$
for all $\lambda \leq 0$, contradicting the uniqueness of \eqref{unique0}.
\end{proof}

Let $\mathcal{F}_n$ denote the set of all semimetrics, $d: [n] \times [n] \rightarrow \mathbb{R}$, on $n$ points and let $\mathcal{P}_n$ denote the set of all projective $n$-polytopes $X \subseteq \ft^n$ containing $\underline{0}$ in the interior. Given $d \in \mathcal{F}_n$, let $D$ denote the matrix $D=(-d(i,j))$ and consider the column space $C(D)$. By Theorem~\ref{semimetricmatrix}, $D$ is a strongly regular idempotent
and $\underline{0}$ is in the interior of $C(D)$. Applying Theorem~\ref{IJKmain} gives that $C(D)$ is a projective $n$-polytope containing the point $\underline{0}$ in its interior. In other words, $C(D) \in \mathcal{P}_n$. Thus we may define a map $\chi: \mathcal{F}_n \rightarrow \mathcal{P}_n$ by $d \mapsto C((-d(i,j)))$. We show that this map is a bijection.

\begin{corollary}
\label{semimetric1-1}
There is a bijection between semimetrics on $n$ points and projective $n$-polytopes containing the point $\underline{0}$ in their interior, given by $$\chi : d \mapsto C((-d(i,j))).$$
\end{corollary}

\begin{proof}
Suppose first that $d_1, d_2 \in \mathcal{F}_n$ with $d_1 \neq d_2$. We shall show $\chi(d_1) \neq \chi(d_2)$. Let $D_1 =  (-d_1(i,j))$ and $D_2 = (-d_2(i,j))$. Then $\chi(d_1) = C(D_1)$ and $\chi(d_2) = C(D_2)$. By Theorem ~\ref{semimetricmatrix}, $D_1$ and $D_2$ are strongly regular idempotents. Moreover, since $d_1 \neq d_2$ we have $D_1 \neq D_2$. Thus, by Theorem~\ref{numidmpt}, $C(D_1) \neq C(D_2)$.

Now let $X \in \mathcal{P}_n$. Thus $X$ is a projective $n$-polytope in $\ft^n$ containing the point $\underline{0}$ in its interior. By Theorem~\ref{IJKmain}, $X=C(D)$ for some idempotent $D \in M_n(\ft)$. Moreover, since $X$ has dimension $n$, $D$ must be strongly regular. Now let $d: [n] \times [n] \rightarrow \mathbb{R}$ be the map given by $d(i,j) = -D_{i,j}$. Since $X=C(D)$ contains $\underline{0}$ in its interior, it follows from Theorem~\ref{semimetricmatrix} that $d \in \mathcal{F}_n$ and hence $\chi(d)=C(D)=X$.
\end{proof}

\begin{lemma}
\label{symmetric}
Let $D$ be a symmetric idempotent matrix in $M_n(\ft)$. Then the following statements are equivalent:
\begin{itemize}
\item[(i)] $D$ is strongly regular;
\item[(ii)] $D$ has a $0$ in each diagonal position, but nowhere else;
\item[(iii)] $D$ has a $0$ in each diagonal position and negative entries off the diagonal.
\end{itemize}
\end{lemma}

\begin{proof}
If $D$ is strongly regular then, by Lemma~\ref{extremalsandzeros}, all
diagonal entries of $D$ are equal to $0$. Suppose for contradiction that $D$ has a zero off the diagonal. Since $D$ is symmetric we have $D_{i,j} = D_{j,i} = 0$ for some $i \neq j$. But then, by Lemma~\ref{fullrank}, $D$ has rank strictly less than $n$, contradicting our assumption that $D$ is strongly regular.

Now suppose $D$ has a $0$ in each diagonal position, but nowhere else. We show all off-diagonal entries must be negative. Indeed, suppose for contradiction that $D_{i,j}>0$. Since $D$ is a symmetric idempotent this gives
$$D_{i,i} = (D^2)_{i,i} \geq D_{i,j} \otimes D_{j,i} = D_{i,j} + D_{i,j} >0,$$
contradicting $D_{i,i}=0$.

Finally, suppose $D$ has a $0$ in each diagonal position and negative entries off the diagonal. Then the permanent of $D$ is achieved uniquely by the identity permutation. Thus $D$ has (tropical) rank $n$, and hence $D$ is strongly regular.
\end{proof}

\begin{theorem}
\label{metricmatrix}
Let $n$ be a positive integer. Given a function $$d:[n] \times [n] \rightarrow \mathbb{R},$$ let $D$ denote the $n \times n$ matrix given by $D = (-d(i,j))$. Then the following statements are equivalent:
\begin{itemize}
\item[(i)] $d$ is a metric;
\item[(ii)] $D$ is a strongly regular symmetric idempotent;
\item[(iii)] $D= D^*=D^T$ with negative entries off the diagonal.
\end{itemize}
\end{theorem}
\begin{proof}
That (i) and (iii) are equivalent follows immediately from Theorem~\ref{semimetricmatrix}. That (i) and (ii) are equivalent follows from Theorem~\ref{semimetricmatrix}
and Lemma~\ref{symmetric}.
\end{proof}

\section{Residuation, the Hilbert metric and tropical realisations}\label{sec_residuation}
In this section we explore how the representation of semimetrics [respectively, metrics] by tropical idempotents yields realisations in tropical space by residuation maps [respectively, the tropical Hilbert metric].

For $x, y \in \ft^n$ we define
$$\langle x \mid y \rangle = \max \lbrace \lambda \in \ft \mid \lambda \otimes x \leq y \rbrace = \min \lbrace y_i - x_i \rbrace.$$
The map $(x, y) \mapsto \langle x \mid y \rangle$ is a \textit{residuation operator} in the sense of
residuation theory \cite{Blyth72}, and is ubiquitous in tropical mathematics. We define a function
$$\delta : \ft^n \times \ft^n \to \mathbb{R} \text{ by } \delta(x,y) = - \langle x \mid y \rangle = \max \lbrace x_i - y_i \rbrace,$$
which we call \textit{residuation distance}. This function already has some
of the natural properties of a distance function:

\begin{proposition}\label{prop_resproperties}
For all $x,y,z \in \ft^n$ we have $\delta(x,x) = 0$ and $$\delta(x,z) \leq \delta(x,y) + \delta(y,z).$$
\end{proposition}
\begin{proof}
Clearly, for any $x \in \ft^n$ we have
$$\delta(x,x) = - \langle x \mid x \rangle = - \max \lbrace \lambda \mid \lambda \otimes x \leq x \rbrace = 0.$$
Also, for any $x,y,z \in \ft^n$ we have
\begin{align*}
-(\delta(x,y) + \delta(y,z)) \otimes x
&= (\langle y \mid z \rangle \otimes \langle x \mid y \rangle) \otimes x \\
&= \langle y \mid z \rangle \otimes (\langle x \mid y \rangle \otimes  x) \\
&\leq \langle y \mid z \rangle \otimes y \\
&\leq z,
\end{align*}
giving
$$-\left( \delta(x,y) + \delta(y,z) \right) \leq \max \lbrace \lambda \mid \lambda \otimes x \leq z \rbrace = \langle x \mid z \rangle = - \delta(x,z).$$
Thus
$$\delta(x,y) + \delta(y,z) \geq \delta(x,z).$$
\end{proof}

However, $\delta$ as defined is not in general a semimetric on $\ft^n$,
since it may take negative values, or give a distance of $0$ between
distinct points. In fact, it is easy to characterise those subsets of
$\ft^n$ on which $\delta$ does define a semimetric. Recall that an
\textit{antichain} in a partial order is a subset in which no two elements
are comparable.

\begin{proposition}\label{prop_antichain}
Let $X \subseteq \ft^n$. Then the residuation distance restricts to a
semimetric on $X$ if and only if $X$ is an antichain.
\end{proposition}
\begin{proof}
First notice that for any $x$ and $y$ in $\ft^n$ we have $\delta(x,y) \leq 0$ if and only if $x = 0x \leq y$. So if $X$ is an antichain then we will have $\delta(x,y) > 0$ provided $x \neq y$ so that $\delta$ restricts to a semimetric on $X$. Conversely, if $X$ is not an antichain then we may choose distinct $x,y \in \ft^n$ such that $x \leq y$, whereupon $\delta(x,y) \leq 0$ so $\delta$ does not restrict to a semimetric on $X$.
\end{proof}

In view of Proposition~\ref{prop_antichain}, if $X \subseteq \ft^n$ is an antichain we use the term \textit{residuation semimetric} for the restriction of residuation distance to $X$.

Proposition~\ref{prop_antichain}, then, tells us that every antichain in $\ft^n$ yields a semimetric
space. It is very natural to ask exactly \textit{which} semimetric spaces arise in this way, that is,
which semimetrics spaces can be represented by residuation in tropical $n$-space. It turns out that
every finite semimetric space is representable in this way. We shall prove this using the results of
the previous sections, via an interesting connection between the residuation operator and idempotency.
Of course there are also infinite (even uncountable) antichains in $\ft^n$;
it remains an interesting open problem to characterise those infinite
semimetric spaces which are representable in $\ft^n$.

\begin{lemma}
\label{residuationbound}
Let $E$ be an idempotent element of $M_n(\ft)$. Let $r_1, \ldots, r_n$ denote the rows of $E$ and $c_1, \ldots, c_n$ denote the columns of $E$. Then
\begin{itemize}
\item[(i)] $E_{i,j} \leq {\rm min} (\langle r_j | r_i \rangle, \langle c_i | c_j \rangle)$ for all $i$ and $j$.
\item[(ii)] Moreover, if $E_{j,j} = 0$, then $E_{i,j} = \langle r_j | r_i \rangle =  \langle c_i | c_j \rangle$ for all $i$.
\end{itemize}
\end{lemma}
\begin{proof}
(i) The equation $E \otimes E = E$ yields
\begin{eqnarray*}
r_i = \bigoplus_{j=1}^{n} E_{i,j} \otimes r_j \text{ and }
c_i = \bigoplus_{j=1}^{n} E_{j,i} \otimes c_j,
\end{eqnarray*}
for all $i$. Thus $r_i \geq E_{i,j} \otimes r_j$ for all $i$ and $j$, giving $E_{i,j} \leq \langle r_j | r_i \rangle$. Similarly, $c_j \geq E_{i,j} \otimes c_i$ for all $i$ and $j$, giving $E_{i,j} \leq \langle c_i | c_j \rangle$.

(ii) Suppose that $E_{j,j} = 0$ and, for a contradiction, that $E_{i,j} < \langle r_j | r_i \rangle$. Then
$$\langle r_j | r_i \rangle + r_{j,j} = \langle r_j | r_i \rangle + E_{j,j} > E_{i,j} + E_{j,j} = E_{i,j} = r_{i,j},$$
contradicting $\langle r_j | r_i \rangle \otimes r_{j} \leq r_i$. A similar argument holds for columns.
\end{proof}

In particular, Lemma~\ref{residuationbound} yields the following fact about
idempotents with all diagonal entries equal to $0$ (including, for example,
all strongly regular idempotents). This may be of independent interest.
\begin{theorem}
\label{thm_idempotentswithzeros}
Let $E \in M_n(\ft)$ be an idempotent matrix, with all diagonal entries equal to $0$ (for example, any idempotent matrix of rank $n$). Suppose the rows of $E$ are $r_1, \dots, r_n$, and the columns of $E$ are $c_1, \dots, c_n$. Then for all $i$ and $j$, $$E_{i,j} = \langle r_j \mid r_i \rangle = \langle c_i \mid c_j \rangle.$$
\end{theorem}

\begin{corollary}\label{cor_matrixrows}
Let $d:[n] \times [n] \rightarrow \mathbb{R}$ be a semimetric and let
$D$ denote the $n \times n$ matrix given by $D = (-d(i,j))$. The
finite semimetric space represented by $D$ is isometric to the subset
of $\ft^n$ consisting of the columns of $D$ under residuation
distance (and hence the columns of $D$ form an antichain in $\ft^n$).
\end{corollary}

\begin{proof}
By Lemma~\ref{triangle}, $D$ is an idempotent with all diagonal entries
equal to $0$. Thus by Theorem~\ref{thm_idempotentswithzeros} we have
$D_{ij} = \langle c_i \mid c_j \rangle$,  where $c_1, \dots, c_n$ denote
the columns of $D$. Thus it is immediate that
$$d(i,j) \ = \ -D_{i,j} \ = \ -\langle c_i \mid c_j \rangle \ = \ \delta(c_i, c_j).$$
\end{proof}

\begin{corollary}
\label{semimetricbracket}
Every semimetric space on $n$ points occurs as the residuation distance semimetric on an antichain of points in $\ft^n$. Moreover, the points may be chosen to be extremal points of a projective $n$-polytope in $\ft^n$ containing $\underline{0}$ in its interior.
\end{corollary}
\begin{proof}
Let $d:[n] \times [n] \rightarrow \mathbb{R}$ be a semimetric and let $D$
denote the $n \times n$ matrix given by $D = (-d(i,j))$. By
Corollary~\ref{cor_matrixrows}, this semimetric space can be realised as the
columns of $D$ with respect to the residuation distance. We claim the columns
of $D$ are extremal points of a projective $n$-polytope in $\ft^n$
containing $\underline{0}$ in its interior. Since $D$ is the matrix of
a semimetric, it follows from Theorem~\ref{semimetricmatrix} that $D$ is a
strongly regular idempotent with $\underline{0}$ in the interior of its
column space. The fact that $D$ is idempotent yields that $C(D)$ is projective (by Theorem~\ref{IJKmain}), whilst strong regularity yields that $C(D)$ has tropical dimension $n$. In other words, $C(D)$ is a projective $n$-polytope. Moreover, it is easy to see that the columns of $D$ are precisely the extremal points of $C(D)$, considered up to scaling. This completes the proof.
\end{proof}

We can define a distance function on $\ft^n$ using the classical mean (or tropical geometric mean) of the two residuation distances between two points:
$$d_H \ : \ \ft^n \times \ft^n \to \mathbb{R}, \ (x,y) \mapsto \frac{1}{2}\left(\delta(x,y) + \delta(y,x) \right).$$
This function is called the \textit{(tropical) Hilbert projective
metric}\footnote{In fact it is more common to define the metric to be the
(classical) sum of the two residuation distances, rather than the mean; the
distinction is basically immaterial but for our purposes considering the
means makes things a little neater.}, and
is widely used in tropical mathematics (see for example \cite{Cohen04,Develin04,K_tropd}).

It is immediate from the definition that $d_H$ is symmetric. It follows easily from the definition of
residuation that $d_H$ is non-negative and from Proposition~\ref{prop_resproperties} that it satisfies the
triangle inequality. This map is not quite a metric on $\ft^n$ since we have
$d_H(x, y) = 0$ if and only if $y = \lambda x$ for some $\lambda \in \ft$.
However, it induces a metric on tropical projective space $\pft^{n-1}$, and
is also a metric when restricted to any antichain in $\ft^n$.

\begin{corollary}
\label{cor_metric_rows}
Let $d:[n] \times [n] \rightarrow \mathbb{R}$ be a metric and let $D$ denote the $n \times n$ matrix given by $D = (-d(i,j))$. The finite metric space represented by $D$ is isometric to the subset of $\ft^n$ consisting of the columns [equivalently, rows] of $D$ with respect to the Hilbert projective metric.
\end{corollary}

\begin{proof}
By Theorem~\ref{thm_idempotentswithzeros}, $D_{i,j} = \langle c_i \mid c_j \rangle$, where $c_1, \dots, c_n$ denote the columns of $D$. Now, since $D$ is symmetric, we obtain
$$d_H(c_i, c_j) = -\frac{1}{2}(\langle c_i | c_j \rangle + \langle c_j | c_i \rangle)= -\frac{1}{2}(D_{i,j} + D_{j,i}) = -D_{i,j} = d(i,j).$$
\end{proof}

\begin{corollary}
\label{metricbracket}
Every $n$-point metric space is realised as the extremal points (considered up to scaling) of a projective $n$-polytope in $\ft^n$ containing $\underline{0}$ in its interior with respect to the Hilbert projective metric.
\end{corollary}

\begin{proof}
Let $d:[n] \times [n] \rightarrow \mathbb{R}$ be a metric and let $D$ denote the $n \times n$ matrix given by $D = (-d(i,j))$. By Corollary~\ref{cor_metric_rows}, this metric space can be realised as the columns of $D$ with respect to the Hilbert projective metric and, by the same line of
reasoning as in the proof of Corollary~\ref{semimetricbracket}, the columns
of $D$ are the extremal points (considered up to scaling) of a pure $n$-dimensional tropical polytope in $\ft^n$ containing $\underline{0}$ in its interior.
\end{proof}

\section{Duality and Symmetry}\label{sec_duality}
By Theorem~\ref{numidmpt}, there is a natural one-to-one correspondence
between projective $n$-polytopes in $\ft^n$ and strongly regular idempotents.
In Section~\ref{semi} we exhibited a natural one-to-one correspondence
between semimetrics
on $n$-points and projective $n$-polytopes in $\ft^n$ containing
$\underline{0}$ in their interior. Given such a polytope, we might
ask if there is a way to see in the polytope whether the corresponding
semimetric is a metric, or equivalently, whether the corresponding
idempotent is symmetric.

Given a matrix $A \in M_n(\ft)$, the \textit{duality map} of $A$ is
$$\theta_A \ : \ R(A) \rightarrow C(A), \ x \mapsto A \otimes (-x),$$
where we ignore the distinction between row and column vectors to minimise the notation.
The duality map is widely used in tropical mathematics (see for example
\cite{Cohen04,Develin04,K_tropd}). It is a canonical bijection between the
row space and the column space of $A$ and although it is not a linear isomorphism,
it preserves a remarkable amount of geometric and order-theoretic structure.
For example, it is an antitone lattice isomorphism \cite{Cohen04}, a
combinatorial isomorphism of Euclidean polyhedral complexes \cite{Develin04},
a homeomorphism and an isometry with respect to the Hilbert projective
metric \cite{K_tropd}, and an \textit{anti-isomorphism} (in a certain
sense) of modules \cite{K_tropd}.

\begin{theorem}\label{thm_idptrowcolumn}
Let $E \in M_n(\ft)$ be a strongly regular idempotent. Then $C(E) = - R(E)$,
and $\theta_E(x) = -x$ for all $x \in R(E)$.
\end{theorem}
\begin{proof}
Let $x \in R(E)$ be an interior point of $R(E)$. Since $\theta_E$ induces an
isometry, $\theta_E(x) = E \otimes (-x)$ must also be an interior point of
$C(E)$. Now, since $E$ is strongly regular idempotent, left multiplication by
$E$ fixes $C(E)$ and by Lemma~\ref{exterior} maps everything outside
$C(E)$ onto the boundary of $C(E)$. Thus $(-x)$ must be an interior
point of $C(E)$ and hence $E\otimes(-x) = (-x)$. So, regarding $R(E)$
and $C(E)$ as subspaces of the same space $\ft^n$, we have shown that the
duality map $\theta_E$ restricted to interior points of $R(E)$ is merely
negation.

Since $E$ is strongly regular we have, by Theorem~\ref{IJKmain}, that
$R(E)$ has pure dimension $n$. It follows that $R(E)$ is the closure of
its interior. Thus, every point of $R(E)$ is a limit of interior points
and so, by continuity of both the duality map and negation,
$\theta_E(x) = -x$ on the
whole of $R(E)$. Since $\theta_E$ is a bijection from $C(E)$ to $R(E)$
this shows that $C(E) = - R(E)$.
\end{proof}

\begin{theorem}
\label{negationclosed}
Let $X \subseteq \ft^n$ be a projective $n$-polytope. Then $X$ is the row
space of a symmetric idempotent if and only if $X = -X$.
\end{theorem}
\begin{proof}
Since $X$ is a projective $n$-polytope in $\ft^n$, Theorem~\ref{numidmpt} tells us that there is a unique idempotent $E$ such that $X$ is the row space of $E$ and it is clear that this idempotent has rank $n$. Suppose first that $E$ is symmetric. Then, by Theorem~\ref{thm_idptrowcolumn}, we have
$$X = R(E) = C(E) = -R(E) = -X.$$

Suppose now that $X=-X$. By Theorem~\ref{thm_idptrowcolumn}, $X=R(E)=-C(E)$ so that $-X= C(E)$. Since $X=-X$ this yields  $R(E) = C(E)= R(E^T)$. Notice that $E^T$ is also idempotent. But $E$ is the unique idempotent with $R(E)=X$,  so we must have $E = E^T$, that is, $E$ is symmetric.
\end{proof}

\begin{theorem}\label{thm_negclose}
There is a bijection between $n$-point metrics and negation-closed projective $n$-polytopes, given by $\chi : d \mapsto C((-d(i,j)))$.
\end{theorem}

\begin{proof}
By Corollary~\ref{semimetric1-1}, the map $\chi: d \mapsto C((-d(i,j)))$ is a bijection between $n$-point semimetrics and projective $n$-polytopes containing $\underline{0}$ in their interior.  We consider the restriction of $\chi$ to the set of $n$-point metrics. Let $d$ be a metric on
$[n]$ and let $D=(-d(i,j))$. Then $\chi(d) = C(D)$ is a projective $n$-polytope containing the point $\underline{0}$ in its interior. By Theorem~\ref{metricmatrix}, $D$ is a symmetric idempotent. Thus Theorem~\ref{negationclosed} gives that $C(D)$ is negation closed. Hence $\chi$ maps each $n$-point metric to a negation-closed projective $n$-polytope.

Now let $X$ be a negation-closed projective $n$-polytope. By Theorem~\ref{negationclosed}, $X$ is the column space of a strongly regular symmetric idempotent matrix $D$. Now let $d: [n] \times [n] \rightarrow \mathbb{R}$ be defined by $d(i,j)=-D_{i,j}$.
Theorem~\ref{metricmatrix} gives that $d$ is a metric and it is clear that
$\chi(d) = C(D)$.
\end{proof}

We digress briefly to mention an application in semigroup theory. Recall
that on any monoid $M$, equivalence relations $\GreenR$ and $\GreenH$ may
be defined by $a \GreenR b$ if $aM = bM$ and $a \GreenH b$ if $aM = bM$
and $Ma = Mb$. An element $a \in M$ is called \textit{von Neumann regular} if
there is an element $b \in M$ such that $aba = a$. It is a standard fact
of semigroup theory that an element is von Neumann regular if and only if
it is $\GreenR$-related to an idempotent, and that the $\GreenH$-classes of
idempotents are exactly the maximal subgroups. In the case of $M_n(\ft)$ (with respect to matrix multiplication),
two matrices are $\GreenR$-related if they have the same column space, and
$\GreenH$-related if they have the same column space and the same row
space \cite{K_tropicalgreen}. Theorem~\ref{numidmpt} (or \cite[Theorem 5.7]{K_puredim})
guarantees that a von Neumann regular matrix of maximal rank is $\GreenR$-related to a
unique idempotent, and hence corresponds to a unique maximal subgroup.
Theorem~\ref{thm_idptrowcolumn} thus gives an explicit description of
the matrices which comprise that subgroup.
\begin{corollary}
Let $M$ be a von Neumann regular matrix of rank $n$. Then the unique
maximal subgroup in the $\GreenR$-class of $M$ is the set of all matrices
$N$ with $C(N) = C(M)$ and $R(N) = -C(M)$.
\end{corollary}

\section{Isometries and Maximal Subgroups} \label{sec_groups}

In recent work \cite{K_tropgroups} the present authors together with Zur Izhakian showed that every maximal subgroup of $M_n(\ft)$ is isomorphic to a group of the form $\mathbb{R} \times \Sigma$, where $\Sigma$ is a subgroup of the symmetric group $S_n$. In particular, let us consider the
case where $D$ is an idempotent matrix in $M_n(\ft)$ corresponding to a \emph{metric} on $n$ points, and let $H_D$ denote the corresponding $\GreenH$-class. We shall show that in this case the associated finite group $\Sigma$ is isomorphic to the isometry group of the finite metric space that we started with. To this end we shall require some additional notation and some results from \cite{K_tropgroups}.

Let $\trop =\ft\cup\{-\infty\}$ and extend the definitions of $\oplus$ and $\otimes$, so that $a \oplus -\infty = -\infty \oplus a = a$ and $a \otimes -\infty = -\infty \otimes a = -\infty$, for all $a \in \trop$. Consider the monoid $M_n(\trop)$ consisting of all $n \times n$ matrices with entries in $\trop$, with respect to the matrix multiplication induced from the operations on $\trop$. It is well known that the \emph{units} in $M_n(\trop)$ are the matrices that contain precisely one entry from $\ft$ in every row and every column. Thus every invertible matrix in $M_n(\trop)$ can be written as the product of a \emph{tropical diagonal matrix} (that is, a matrix with entries from $\ft$ on the diagonal and $-\infty$ entries elsewhere) and a \emph{tropical permutation matrix} (that is, a matrix with precisely one zero entry in every row and every column and all other entries equal to $-\infty$). Now let $E$ be an idempotent in $M_n(\ft)$. It was shown in \cite[Theorem 7.3]{K_tropgroups} that the $\GreenH$-class of $E$ is isomorphic to the group consisting of all units commuting with $E$.

\begin{lemma}
\label{lem_isoms}
Let $d:[n] \times [n] \rightarrow \mathbb{R}$ be a metric and let $D$ denote the $n \times n$ matrix given by $D = (-d(i,j))$. A permutation matrix commutes with $D$ if and only if the corresponding permutation is an isometry of the metric space $([n],d)$.
\end{lemma}

\begin{proof}
The isometries of $([n],d)$ are precisely those permutations $\sigma \in S_n$ such that $d(\sigma(i), \sigma(j)) = d(i,j)$ for all $i$ and $j$. In other words, they are the permutations $\sigma$ such that such that $D_{\sigma(i),\sigma(j)} = D_{i,j}$. It follows immediately that the isometries of $([n],d)$ correspond exactly to the permutation matrices $P$ satisfying $PDP^{-1} = D$.
\end{proof}

\begin{lemma}
\label{lem_perms}
Let $d:[n] \times [n] \rightarrow \mathbb{R}$ be a metric and let $D$ denote the $n \times n$ matrix given by $D = (-d(i,j))$. If $G \in M_n(\trop)$
is a unit commuting with $D$ then $G = \lambda \otimes P$ for some $\lambda \in \ft$ and some permutation matrix $P$.
\end{lemma}

\begin{proof}
Let $D_1, \ldots, D_n$ denote the columns of $D$. Since $G$ commutes with $D$, it follows from \cite[Theorem 3.4 and Theorem 7.3]{K_tropgroups} that left multiplication by $G$ restricts to a linear automorphism of the column space of $D$. It is readily verified that linear isomorphisms preserve the Hilbert metric, so that the left action of $G$ induces an isometry on the finite metric space consisting of the columns of $D$ with respect to $d_H$. In other words,
$$d_H(G\otimes D_i, G \otimes D_j) = d_H(D_i, D_j),$$
for all $i$ and $j$.

Now, $G$ is a unit, so we can write $G=SP$, where $S$ is a diagonal matrix,
say with entries $S_{i,i} = \lambda_i$, and $P$ is a permutation matrix,
say with $P_{\sigma(i), i} = 0$ and $P_{j,i} = -\infty$ for $j \neq \sigma(i)$ for some $\sigma \in S_n$. Since $GD = DG$, for every $i$ we have
$$G \otimes D_i = (GD)_i = (DG)_i = \lambda_i D_{\sigma(i)}.$$
It then follows from the fact that the Hilbert metric is defined on projective space that for every $i$ and $j$,
\begin{eqnarray*}
d_H(D_i, D_j)&=& d_H(G \otimes D_i, G\otimes  D_j)\\
&=& d_H(\lambda_i D_{\sigma(i)}, \lambda_j D_{\sigma(j)})\\
&=& d_H(D_{\sigma(i)}, D_{\sigma(j)}).
\end{eqnarray*}

By Corollary~\ref{cor_metric_rows}, the finite metric space consisting of the columns of $D$ with respect to $d_H$ is isometric to $([n], d)$, giving
$$d(i,j) = d_H(D_i, D_j) = d_H(D_{\sigma(i)}, D_{\sigma(j)}) = d(\sigma(i), \sigma(j)),$$
for all $i$ and $j$. Thus $\sigma$ is an isometry of $([n], d)$ and, by Lemma~\ref{lem_isoms}, $PD = DP$.

But if $P$ and $G$ both commute with $D$, then so does the diagonal matrix $S = G P^{-1}$. Clearly the only diagonal matrices that commute with
$D$ (or any finite matrix) are the scalar matrices, so we conclude that $S = \lambda I_n$ for some $\lambda \in \ft$, and hence $G = \lambda P$.
\end{proof}

\begin{theorem}
\label{thm_isomgp}
Let $d:[n] \times [n] \rightarrow \mathbb{R}$ be a metric and let $D$ denote the $n \times n$ matrix given by $D = (-d(i,j))$. Then the $\GreenH$-class of $D$ is isomorphic to $I \times \mathbb{R}$, where $I$ is the isometry group of the finite metric space $([n],d)$.
\end{theorem}

\begin{proof}
Let $H_D$ denote the $\GreenH$-class of $D$. By \cite[Theorem 7.9]{K_tropgroups}, $H_D \cong \mathbb{R} \times \Sigma$, where $\Sigma$ is the group of all units $G$ such that  $G$ has eigenvalue 0 and $G$ commutes with $E$. It follows immediately from Lemma~\ref{lem_perms} that $\Sigma$ is the set of all permutation matrices commuting with $E$. Thus, by Lemma~\ref{lem_isoms}, we see that $\Sigma \cong I$.
\end{proof}

\begin{corollary}\label{cor_allsubgroups}
Every group of the form $G \times \mathbb{R}$ with $G$ finite occurs as a maximal subgroup of $M_n(\ft)$ for sufficiently large $n$.
\end{corollary}
\begin{proof}
By a result of Asimov \cite{Asimov76}, every finite group $G$ is isomorphic
to the isometry group of a finite metric space; in fact, it is shown
that if $G$ has cardinality $k$, there exist $k(k-1)$ points in Euclidean
$(k-1)$-space such that $G$ is isomorphic to the isometry group of this
finite metric space. Let $n=k(k-1)$ and let $D$ denote the $n \times n$
idempotent tropical matrix corresponding to this finite metric. By
Theorem~\ref{thm_isomgp}, the corresponding maximal subgroup of $M_n(\ft)$
is isomorphic to $G \times \mathbb{R}$.
\end{proof}

In \cite{K_tropgroups}, Izhakian and the authors showed that every subgroup of $M_n(\ft)$ is isomorphic to $G \times \mathbb{R}$ for
some finite group $G$. Combining this with Corollary~\ref{cor_allsubgroups} yields:

\begin{corollary}
The maximal subgroups of the semigroups of the form $M_n(\ft)$ are exactly the groups of the form $G \times \mathbb{R}$ for $G$ finite.
\end{corollary}

We note that our results do not quite completely classify the maximal subgroups of each individual full matrix semigroup $M_n(\ft)$. If
$G$ is a finite group, then the smallest $n$ such that $G \times \mathbb{R}$ occurs a maximal subgroup of $M_n(\ft)$ is bounded below by the
permutation degree of $G$ (by \cite{K_tropgroups}) and bounded above by $|G|(|G|-1)$ (from the proof of Corollary~\ref{cor_allsubgroups}). It
is an interesting open question whether these bounds are tight or if (as we suspect) this rather large gap can be narrowed.

\section{Examples}\label{sec_examples}
In this section we study in detail a few elementary examples of projective tropical
polytopes in low dimension, and show how the concepts and results of this
paper apply to them.

We consider first polytopes in $\ft^2$. The two-dimensional case is very
much degenerate and our results specialised to this case can be obtained
by simpler means; nevertheless it still informative to see how the
results manifest themselves.
It follows from the results of \cite{K_tropicalgreen} and \cite{K_puredim}
that every two-dimensional tropical polytope is projective. Projective
polytopes (indeed, two-dimensional tropical polytopes in general), do not
display the ``dimension anomaly'' which appears
in higher dimensions: every polytope in $\ft^2$ is thus either a projective
$1$-polytope or a projective $2$-polytope. The projective $1$-polytopes
are exactly the (classical) lines of gradient 1 in the plane, while the
projective $2$-polytopes are closed connected
regions bounded by lines of gradient 1. These are shown in Figure~\ref{fig2d}.

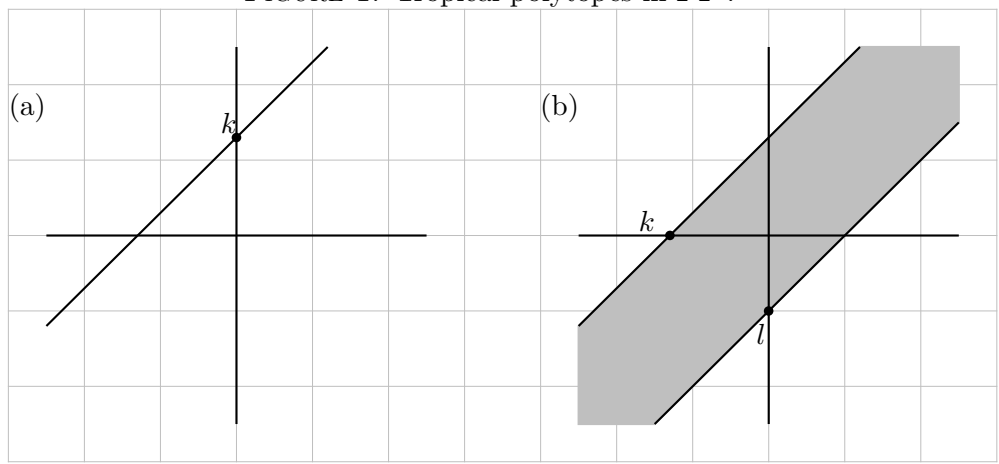
\begin{figure}[h]
\caption{Tropical polytopes in $\ft^{2}$.}
\begin{pspicture}(13,6)
\psgrid[gridcolor=lightgray, gridwidth=0.25pt, gridlabels=0pt, subgridwidth=0.25pt, subgriddiv=1](0,0)(13,6)
\rput(0.25,4.7){(a)}
\qline(0.5,3)(5.5,3)
\qline(3,0.5)(3,5.5)
\qline(0.5,1.8)(4.2,5.5)
\rput(2.9,4.5){$k$}
\psdots*(3,4.3)
\rput(7.25,4.7){(b)}
\pspolygon[fillstyle=solid, fillcolor=lightgray, linecolor=lightgray](7.5,1.8)(7.5,0.5)(8.5,0.5)(12.5,4.5)(12.5,5.5)(11.2,5.5)
\qline(7.5,3)(12.5,3)
\qline(10,0.5)(10,5.5)
\qline(7.5,1.8)(11.2,5.5)
\qline(8.5,0.5)(12.5,4.5)
\psdots*(8.7,3)(10,2)
\rput(8.4,3.2){$k$}
\rput(9.9,1.7){$l$}
\end{pspicture}
\label{fig2d}
\end{figure}

Theorem~\ref{numidmpt} describes the number of idempotents having a given polytope as column space.
Each projective $1$-polytope in $\ft^2$ is the column space of continuum-many distinct idempotents; these are of less interest to us here, but a complete description can be found in \cite{K_tropicalgreen}.
Each projective $2$-polytope, by contrast, is the column space of a unique idempotent, and it is these with
which we are primarily concerned. These idempotents have the form
$$E = \left(\begin{array}{cc}
0&k\\
l&0
\end{array}\right).$$
where $k+l < 0$. In terms of the corresponding polytope $C(E)$, $k$ and $l$ are the $x$-intercept
of the upper boundary and $y$-intercept of the lower boundary, as marked in Figure~\ref{fig2d}.

This matrix corresponds to the asymmetric distance function on the set $\lbrace 1,2 \rbrace$ given by $d(1,1) = d(2,2) = 0$,
$d(1,2) = -k$ and $d(2,1) = -l$. Note that the condition $k+l < 0$ (necessary to ensure that the matrix
is idempotent of rank $2$) ensures that this function satisfies the triangle inequality. It will be
a semimetric provided $k, l < 0$; geometrically this can be seen to happen exactly if the origin
$(0,0)$ lies in the interior of $C(E)$, as expected by Theorem~\ref{semimetricmatrix}.
The function will be a metric if in addition $k=l$. Geometrically, this can be seen to happen if
$C(E)$ has rotational symmetry through an angle of $\pi$ around the origin, that is, if $C(E) = -C(E)$,
as predicted by Theorem~\ref{thm_negclose}.

We now consider some examples of polytopes in higher dimensions.

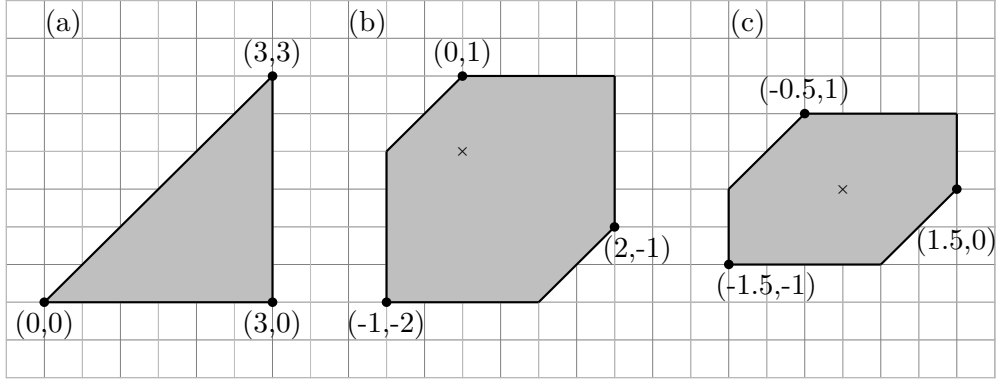
\begin{figure}[h]
\caption{Some tropical polytopes in $\pft^{2}$.}
\begin{pspicture}(14,5)
\psgrid[gridcolor=lightgray, gridwidth=0.25pt, gridlabels=0pt, subgridwidth=0.25pt, subgriddiv=2](0,0)(13,5)
\rput(0.75,4.7){(a)}
\pspolygon[fillstyle=solid, fillcolor=lightgray, linecolor=lightgray](0.5,1)(3.5,1)(3.5,4)
\qline(0.5,1)(3.5,1)
\qline(0.5,1)(3.5,4)
\qline(3.5,1)(3.5,4)
\psdots*(0.5,1)(3.5,1)(3.5,4)
\rput(0.5,0.7){(0,0)}
\rput(3.5,0.7){(3,0)}
\rput(3.5,4.3){(3,3)}
\rput(4.75,4.7){(b)}
\pspolygon[fillstyle=solid, fillcolor=lightgray, linecolor=lightgray](5,3)(6,4)(8,4)(8,2)(7,1)(5,1)
\qline(5,1)(5,3)
\qline(5,3)(6,4)
\qline(6,4)(8,4)
\qline(8,4)(8,2)
\qline(8,2)(7,1)
\qline(7,1)(5,1)
\psdots*(5,1)(6,4)(8,2)
\rput(5,0.7){(-1,-2)}
\rput(6,4.3){(0,1)}
\rput(8.3,1.7){(2,-1)}
\psdot[dotstyle=x](6,3)
\rput(9.75,4.7){(c)}
\pspolygon[fillstyle=solid, fillcolor=lightgray, linecolor=lightgray](9.5,2.5)(10.5,3.5)(12.5,3.5)(12.5,2.5)(11.5,1.5)(9.5,1.5)
\qline(9.5,1.5)(9.5,2.5)
\qline(9.5,2.5)(10.5,3.5)
\qline(10.5,3.5)(12.5,3.5)
\qline(12.5,3.5)(12.5,2.5)
\qline(12.5,2.5)(11.5,1.5)
\qline(11.5,1.5)(9.5,1.5)
\psdots*(9.5,1.5)(10.5,3.5)(12.5,2.5)
\rput(10,1.2){(-1.5,-1)}
\rput(10.5,3.8){(-0.5,1)}
\rput(12.5,1.8){(1.5,0)}
\psdot[dotstyle=x](11,2.5)
\end{pspicture}
\label{fig3poly}
\end{figure}

Figure~\ref{fig3poly} shows the projectivisations of three different projective $3$-polytopes in $\ft^3$. They
clearly have pure dimension $3$, and so by Theorems~\ref{IJKmain} and
\ref{numidmpt}, each is the column
space of a unique idempotent matrix. In fact the three idempotents are
$$\left(\begin{array}{c c c}
0&0&0\\
-3&0&0\\
-3&-3&0
\end{array}\right), \ \
\left(\begin{array}{c c c}
0&-1&-1\\
-3&0&-2\\
-2&-1&0
\end{array}\right) \text{  and  }
\left(\begin{array}{c c c}
0&-1.5&-1.5\\
-1.5&0&-1\\
-1.5&-1&0
\end{array}\right)
$$
respectively.

The origin $(0,0,0)$ lies on the boundary of polytope (a) but in the interiors of polytopes (b) and (c).
Thus, by Theorem~\ref{semimetricmatrix}, polytopes (b) and (c) correspond to semimetrics on three points,
while polytope (a) does not. The values of the semimetrics in question are given by the negating the entries in the
corresponding idempotent matrices; note that the idempotent corresponding to polytope (a) contains $0$ off the main diagonal
and so does not yield a semimetric.

Theorem~\ref{thm_negclose} tells us that polytopes corresponding to \textit{metrics}
must be closed under negation. Geometrically, this means they must have
rotational symmetry through an angle of
$\pi$ around the origin in projective space. Polytope (b) in Figure~\ref{fig3poly} is not closed under negation (as seen from the fact it is not centred around the origin). Polytope (c), on the other hand, is
closed under negation and so by Theorem~\ref{thm_negclose} the associated semimetric is a
metric. This is evident in the symmetry of the corresponding idempotent. In fact the only
projective tropical $3$-polytopes in $\ft^3$ which are negation-closed are (classical) hexagons
and parallelograms centred on the origin, as illustrated in
Figure~\ref{fig_3metric}.

The parallelograms, which are degenerate forms of the hexagon,
arise when the three points in the metric space are collinear, that is, when one of the
distances between pairs is the sum of the other two.

\begin{figure}[h]
\caption{Tropical polytopes in $\pft^{2}$ corresponding to metrics on three points.}
\begin{pspicture}(11,3)
\psgrid[gridcolor=lightgray, gridwidth=0.25pt, gridlabels=0pt, subgridwidth=0.25pt, subgriddiv=2](0,0)(11,3)
\pspolygon[fillstyle=solid, fillcolor=lightgray, linecolor=lightgray](0.5,0.5)(0.5,2.5)(3.5,2.5)(3.5,0.5)
\qline(0.5,0.5)(0.5,2.5)
\qline(0.5,0.5)(3.5,0.5)
\qline(0.5,2.5)(3.5,2.5)
\qline(3.5,2.5)(3.5,0.5)
\psdots(0.5,0.5)(0.5,2.5)(3.5,0.5)
\psdot[dotstyle=x](2,1.5)
\pspolygon[fillstyle=solid, fillcolor=lightgray, linecolor=lightgray](4,0.5)(4,1)(5.5,2.5)(7,2.5)(7,2)(5.5,0.5)
\qline(4,0.5)(4,1)
\qline(4,1)(5.5,2.5)
\qline(5.5,2.5)(7,2.5)
\qline(7,2.5)(7,2)
\qline(7,2)(5.5,0.5)
\qline(5.5,0.5)(4,0.5)
\psdots(4,0.5)(5.5,2.5)(7,2)
\psdot[dotstyle=x](5.5,1.5)
\pspolygon[fillstyle=solid, fillcolor=lightgray, linecolor=lightgray](7.5,0.5)(9.5,2.5)(10.5,2.5)(8.5,0.5)
\qline(7.5,0.5)(9.5,2.5)
\qline(9.5,2.5)(10.5,2.5)
\qline(10.5,2.5)(8.5,0.5)
\qline(8.5,0.5)(7.5,0.5)
\psdots(7.5,0.5)(9.5,2.5)(10.5,2.5)
\psdot[dotstyle=x](9,1.5)
\end{pspicture}
\label{fig_3metric}
\end{figure}
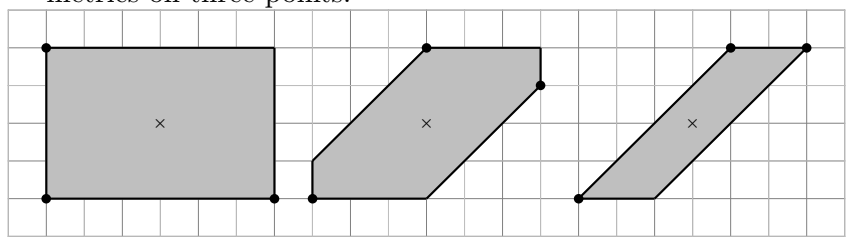

\begin{figure}[h]
\caption{A non-Euclidean metric on four points.}
\begin{pspicture}(4,4.5)
\pspolygon(0,0)(4,0)(2,4)
\pspolygon(0,0)(4,0)(2,1.8)
\pspolygon(0,0)(2,1.8)(2,4)
\rput(-0.25,-0.25){c}
\rput(4.25,-0.25){b}
\rput(2,4.25){a}
\rput(2,1.5){d}
\rput(1,0.7){1}
\rput(3,0.7){1}
\rput(2.1,2.7){1}
\rput(0.8,2){2}
\rput(3.2,2){2}
\rput(2,-0.25){2}
\end{pspicture}
\label{fig_noneuclid}
\end{figure}
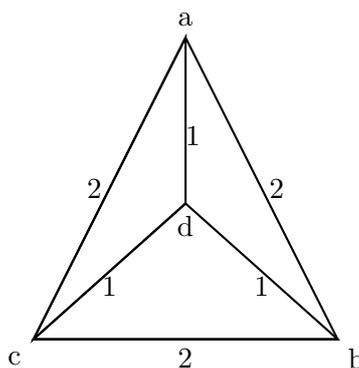

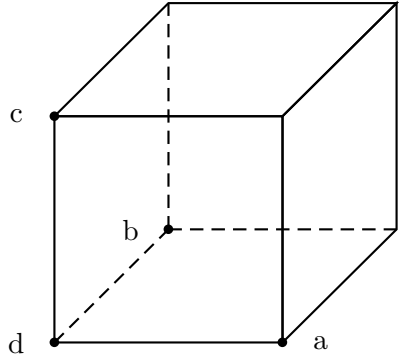
\begin{figure}
\caption{An embedding of the non-Euclidean metric from Figure 4 into $\mathbb{PFT}^3$.}
\begin{pspicture}(-2,-2)(4,3.5)
\psdots(0,0)(-1.5,-1.5)(-1.5,1.5)(1.5,-1.5)
\pspolygon(-1.5,-1.5)(-1.5,1.5)(1.5,1.5)(1.5,-1.5)
\pspolygon(-1.5,1.5)(1.5,1.5)(3,3)(0,3)
\pspolygon(1.5,-1.5)(1.5,1.5)(3,3)(3,0)
\psline[linestyle=dashed](-1.5,-1.5)(0,0)
\psline[linestyle=dashed](0,0)(0,3)
\psline[linestyle=dashed](0,0)(3,0)
\rput(-0.5,0){b}
\rput(-2,-1.5){d}
\rput(-2,1.5){c}
\rput(2,-1.5){a}
\end{pspicture}
\label{fig_cube}
\end{figure}

Another example is given by the $4$-point metric space illustrated schematically in
Figure~\ref{fig_noneuclid}. This clearly cannot be embedded isometrically into Euclidean space
of any dimension. Indeed, the uniqueness of geodesics in Euclidean space would force the image of $d$
under such an embedding to
lie on all three sides of a non-degenerate triangle with vertices the images of $a$, $b$ and $c$,
which is clearly impossible.
However, this metric can be isometrically embedded into $\ft^4$ via the map
\begin{align*}
&a \mapsto (0, -2, -2, -1), \hspace{6ex}  &b \mapsto (-2, 0, -2, -1) \\
&c \mapsto (-2, -2, 0, -1),        & d \mapsto (-1, -1, -1, 0).
\end{align*}
In projective space $\pft^3$, these points are four vertices of a (Euclidean) cube, which in
fact is their tropical convex hull. This is shown in Figure~\ref{fig_cube}.

\section*{Acknowledgements}
This research was supported by EPSRC grant number EP/H000801/1 (\textit{Multiplicative Structure of Tropical Matrix Algebra}).

\bibliographystyle{plain}

\def\cprime{$'$} \def\cprime{$'$}

\end{document}